\documentclass[11pt]{amsart}
\usepackage{amsfonts,amsmath,amscd,latexsym,amssymb,amsthm}
\usepackage[english]{babel}
\usepackage{srcltx} 
\usepackage[all]{xy}
\usepackage{calrsfs}
\usepackage{graphicx}
\usepackage{color}

\newtheorem{thm}{Theorem}[section]
\newtheorem{prop}[thm]{Proposition}
\newtheorem{cor}[thm]{Corollary}

\newtheorem{lem}[thm]{Lemma}

\newtheorem{claim}[thm]{Claim}

\theoremstyle{definition}
\newtheorem{deff}[thm]{Definition}
\newtheorem{rem}[thm]{Remark}
\newtheorem{ex}[thm]{Example}

\newcommand{\sheaf}[1]{\mathcal{#1}}

\newcommand{\E}{\sheaf{E}}
\newcommand{\F}{\sheaf{F}}
\newcommand{\G}{\sheaf{G}}
\newcommand{\I}{\sheaf{I}}

\renewcommand{\O}{\sheaf{O}}
\renewcommand{\P}{\sheaf{P}}

\newcommand{\T}{\mathcal{T}}

\newcommand{\PP}{\mathbb{P}} 
\newcommand{\tensor}{\otimes} 
\newcommand{\onto}{\twoheadrightarrow} 
\newcommand{\iso}{\cong} 
\newcommand{\res}[2]{\left.#1\right|_{#2}} 
\DeclareMathOperator{\Ker}{Ker} 
\DeclareMathOperator{\Spec}{Spec} 
\DeclareMathOperator{\Soc}{Soc} 
\DeclareMathOperator{\supp}{Supp} 
\DeclareMathOperator{\Fitt}{Fitt} 
\DeclareMathOperator{\Pic}{Pic} 

\newcommand{\HHom}{\mathcal{H}\mathit{om}} 
\newcommand{\EExt}{\mathcal{E}\mathit{xt}} 

\newcommand{\set}[1]{\left\{#1\right\}} 
\newcommand{\abs}[1]{\left\vert#1\right\vert} 

\title[Finite subschemes of abelian varieties]{Finite subschemes of abelian varieties and the
Schottky problem}
\author[M. G. Gulbrandsen]{Martin G. Gulbrandsen}
\address{Stord/Haugesund University College, Bj\o{}rnsons gate 45, NO-5528 Haugesund, Norway}
\email{martin.gulbrandsen@hsh.no}

\author[M. Lahoz]{Mart\'{\i} Lahoz}
\address{Departament d'\`Algebra i Geometria, Facultat de Matem\`atiques, Universitat
de Barcelona, Gran Via, 585, 08007 Barcelona, Spain}  \email{{\tt marti.lahoz@ub.edu
}}

\curraddr{Mathematisches Institut, Universit\"at Bonn, Endenicher Allee 60, 53115 Bonn, Germany.}

\thanks{The first author thanks the Royal Institute of Technology in Stockholm (KTH),
where he was a postdoctoral fellow while much of this work was carried out. The second author thanks the Universit\`a di
Roma ``Tor Vergata'' and Giuseppe Pareschi for his kind hospitality during the first steps of this work. He also thanks
the KTH for the one week invitation that allowed to strengthen the collaboration with the first author. The second
author was partially supported by the Proyecto de Investigaci\'on MTM2009-14163-C02-01.}
\begin{document}

\begin{abstract}
The Castelnuovo-Schottky theorem of Pareschi--Popa characterizes
Jacobians, among indecomposable principally polarized abelian varieties
$(A,\Theta)$ of dimension $g$, by the
existence of $g+2$ points $\Gamma\subset A$ in
special position with respect to $2\Theta$, but general with respect to $\Theta$, and furthermore
states that such collections of points must be contained in an Abel-Jacobi
curve. Building on the ideas in the original paper,
we give here a self contained, scheme theoretic proof of the theorem,
extending it to finite, possibly nonreduced subschemes $\Gamma$.
\end{abstract}

\maketitle

\section{Introduction}\label{sec:intro}

There is a classical result in projective geometry, due to
Castelnuovo, saying that a large, but finite, collection of
points in $\PP^r$ which is in linearly general position, but in
sufficiently special position with respect to quadrics, is
contained in a unique rational normal curve.

Pareschi--Popa \cite{PPschottky} have
discovered an analogy for $g$-dimensional indecomposable principally polarized abelian
varieties $(A,\Theta)$, where divisors algebraically equivalent
to $\Theta$ play the role of hyperplanes, and divisors
algebraically equivalent to $2\Theta$ play the role of quadrics.
The Castelnuovo result of Pareschi--Popa says that if we have $g+2$
points on $A$, in a suitable sense general with respect to
$\Theta$, but special with respect to $2\Theta$, then $A$ is the
Jacobian of a curve $C$, and the $g+2$ points are contained in an
Abel-Jacobi curve, i.e.~a translate of $C$, embedded
into its Jacobian. Thus Abel-Jacobi curves play the role of
rational normal curves, and the analogue of Castelnuovo's result
contains a Schottky statement (precise definitions of the terms
occurring in the assumption are given in the next section):

\begin{thm}\label{thm}
Let $\Gamma\subset A$ be a theta-general finite subscheme of degree $g+2$,
imposing less than $g+2$ conditions on general
$2\Theta$-translates. Then the following holds:
\begin{enumerate}
\item \textbf{Schottky:} The indecomposable principally polarized abelian variety $(A,\Theta)$ is
isomorphic to a Jacobian $J(C)$ of a curve $C$, with its canonical
polarization.
\item \textbf{Castelnuovo:} There is an isomorphism $A\iso J(C)$ as above
such that the subscheme $\Gamma$ is contained in an
Abel-Jacobi curve.
\end{enumerate}
\end{thm}

The theorem was proved by Pareschi--Popa \cite{PPschottky} for reduced subschemes $\Gamma$.
The aim of this article is the extension to possibly nonreduced subschemes $\Gamma$, 
as Eisenbud--Harris did in the projective case for the classical Castelnuovo result \cite[Thm. 2.2]{EH2}.
Moreover, in the proof of Pareschi--Popa for reduced subschemes, the scheme structure needed to prove the existence of
trisecants is implicit (see Remark \ref{rem:objection}).
Since we work scheme theoretically from the start, our work may serve to clarify the situation. 

Grushevsky gave also a very similar result \cite[Thm. 3]{grushevsky-err} in the reduced case. His hypothesis are
slightly different and he uses the analytic theory of theta functions to prove it. He uses this result to find equations
for the locus of hyperelliptic Jacobians.

The Schottky and Castelnuovo statements are proved in Sections \ref{sec:schottky} and \ref{sec:castelnuovo},
respectively.
Both results depend on an analysis of theta-duals and dependence loci of subschemes of $\Gamma$;
these concepts are introduced in the preliminary Section \ref{sec:prelim} and further analysed in Section \ref{sec:key}.
The proof of the Schottky result is based on the criterion of Gunning and Welters,
characterizing Jacobians among indecomposable principally polarized abelian varieties by trisecants to the associated
Kummer variety.

The full trisecant conjecture,
which characterizes Jacobians by the existence of a single trisecant,
has now been proved by Krichever \cite[Thm.~1.1]{krichever} over the complex numbers.
Pareschi--Popa have kindly pointed out that the trisecant conjecture, which was not proved when their work appeared,
would simplify their argument.
It seems to us that our version of the argument can not be shortened substantially, as the construction produces a
family of trisecants in any case.
An advantage of using only Gunning--Welters is that our proof of Theorem \ref{thm} is entirely algebraic and valid in
arbitrary characteristic different from $2$ (see \cite[Remark 0.7]{Wcrit}).

The main line of argument is borrowed from Pareschi--Popa,
although many (set theoretic) statements did not seem to translate well into schematic ones,
and had to be substituted.
Moreover, new phenomena occur when $\Gamma$ is nonreduced,
e.g.~already the fact that $\Gamma$ is contained in a nonsingular curve (i.e.~$\Gamma$ is curvilinear) is not obvious.
Logically, our work does not depend on \emph{loc.~cit.}, and can be read independently.

The converse to the theorem is easy, as we
explain in Section \ref{sec:super}: a finite degree $g+2$ subscheme
$\Gamma$ of an Abel-Jacobi curve $C$ imposes less than $g+2$ conditions on
general $2\Theta$-translates in the Jacobian.

\textbf{Acknowledgments} We are grateful to Giuseppe Pareschi, Mihnea Popa and Gerald Welters for valuable discussions.
The second author thanks also Miguel \'Angel Barja and Joan Carles Naranjo for numerous conversations.

\section{Preliminaries}\label{sec:prelim}

Throughout, $(A,\Theta)$ denotes a principally polarized abelian
variety over an algebraically closed field of characteristic different from $2$.
We assume that the divisor
$\Theta$ is symmetric and irreducible. For each point $a\in A$ we
denote by $t_a\colon A\to A$ the translation map. For any subscheme
$Y\subset A$, we write $Y_a$ for the \emph{inverse} image
$Y-a$ under $t_a$.

The \emph{support} $\supp \F$ of a coherent sheaf $\F$ is defined as a subscheme by the annihilator ideal.
We will also make use of the scheme structure, on the same underlying set, defined by the Fitting ideal of $\F$, which
has the advantage of respecting base change.
This will be called the \emph{Fitting support} and denoted $\Fitt \F$.

Inclusions, intersections and unions of subschemes are always to be understood scheme theoretically.
Unions are usually defined by intersection of ideals; we also make use of the scheme structure defined by the product of
ideals, but point this out whenever it is of importance.
When $Y\subseteq X$ are two subschemes of some ambient scheme,
we write $\I_{Y/X} = \I_Y/\I_X$ for the ideal of $Y$ in $\O_X$.

\subsection{The Fourier-Mukai transform}
 
Using the principal polarization $\Theta$, we identify $A$ with
its dual $\Pic^0(A)$.
Thus, the Poincar\'e line bundle $\P$ is identified with the Mumford line bundle $\O_{A\times A}(m^*\Theta - p_1^*\Theta
- p_2^*\Theta)$, where $m$ denotes the group law and
$p_1$ and $p_2$ are the projections. 
The restriction of $\P$ to $A\times\set{a}$ is the homogeneous line bundle $\P_a = \O_A(\Theta_a-\Theta)$.

Following Mukai \cite{mukai},
we define a left exact endofunctor $\mathcal{S}$ on the
category of $\O_A$-modules by
\begin{equation*}
\mathcal{S}(\E) = p_{2*}(p_1^*(\E)\otimes \P).
\end{equation*}
The Fourier-Mukai functor is the total derived functor of $\mathcal{S}$,
and is an autofunctor on the derived category of $A$.
We will not make use of this, and will
just work with $\mathcal{S}$ and its right derived functors $R^i\mathcal{S}$, and
will sloppily refer to these as Fourier-Mukai functors.

\begin{deff}[Mukai \cite{mukai}]
Let $\E$ be an $\O_A$-module.
\begin{enumerate}
\item $\E$ satisfies the \emph{weak index theorem} with index
$i$, abbreviated WIT$_i$, if we have
\begin{equation*}
R^p\mathcal{S}(\E) = 0\quad\text{for all $p\ne i$.}
\end{equation*}
\item $\E$ satisfies the \emph{index theorem} with index
$i$, abbreviated IT$_i$, if we have
\begin{equation*}
H^p(A, \E\otimes\P_a) = 0\quad\text{for all
$a\in A$ and all $p\ne i$.}
\end{equation*}
\end{enumerate}
\end{deff}

\begin{deff}[Mukai \cite{mukai}]
If $\E$ satisfies WIT$_i$, its \emph{Fourier-Mukai
transform} is
\begin{equation*}
\widehat{\E} = R^i\mathcal{S}(\E).
\end{equation*}
\end{deff}

By the base change theorem in cohomology, IT$_i$ implies WIT$_i$,
and the Fourier-Mukai transform of a sheaf satisfying IT is locally free.

We refer to Mukai's original paper \cite{mukai} for the basic properties of the Fourier-Mukai functor.

\subsection{Theta-duality}

Given a morphism $f\colon T\to A$, we define translation $t_f$ along $f$
to be the composite
\begin{equation*}
t_f\colon A\times T \xrightarrow{1\times f} A\times A \xrightarrow{m} A.
\end{equation*}

\begin{deff}
The \emph{theta-dual} of a closed subscheme $Y\subseteq A$ is the unique closed subscheme $T(Y)\subseteq A$
with the universal property that an arbitrary morphism $f\colon T\to A$ factors through $T(Y)$
if and only if $Y\times T \subseteq t_f^{-1}(\Theta)$ inside $A\times T$.
\end{deff}

Thus, as a set, $T(Y)$ consists of those points $a\in A$ for which the theta-translate $\Theta_a$ contains $Y$ as a
scheme.
We want to show that the theta-dual always exists as a scheme,
and equals the object with the same name defined in Pareschi--Popa \cite{PPminimal}.

\begin{lem}\label{lem:zero}
Let $\pi\colon X\to S$ be a projective morphism of schemes, with $S$
noetherian.
\begin{enumerate}
\item Let $\phi\colon \F \to \G$ be a homomorphism of coherent $\O_X$-modules such that $\G$ is $S$-flat.
Then there exists a unique closed subscheme $Z_{\pi}(\phi)\subseteq S$ with the universal property
that an arbitrary morphism $f\colon T\to S$ factors through $Z_{\pi}(\phi)$ if and only if
$f_T^*(\phi)=0$ as a homomorphism of coherent modules on $X_T$.
\item Assume $\pi$ is flat, and let $X'\subseteq X$ be a closed subscheme.
Then there exists a unique closed subscheme $S' \subseteq S$ with the universal property
that an arbitrary morphism $f\colon T\to S$ factors through $S'$ if and only if
$(X')_T = X_T$.
\end{enumerate}
\end{lem}

We call the scheme $Z_{\pi}(\phi)\subseteq S$ in part (1) the relative zero scheme of $\phi$. Before proving the Lemma,
we conclude that theta-duals exist.

\begin{prop}\label{prop:thetadual}
The theta-dual $T(Y)$ exists, for any closed subscheme $Y\subseteq A$.
\end{prop}

\begin{proof}
Apply part (2) of Lemma \ref{lem:zero} to the
subscheme $X' = m^{-1}(\Theta)\cap (Y\times A)$ of $X = Y\times A$,
viewed as schemes over $A$ via second projection.
This gives the result, since $X'_T=X_T$ says that $t_f^{-1}(\Theta)\cap (Y\times T)=Y\times T$.
\end{proof}

\begin{proof}[Proof of Lemma \ref{lem:zero}]
The second part follows from the first:
take $S'\subseteq S$ to be the relative zero scheme of the inclusion $\phi\colon \I_{X'} \subset \O_X$.
Then the condition $f^*_T(\phi)=0$ is equivalent to $\I_{X'_T/X_T}=0$.

For the first part,
replace $\phi\colon \F\to\G$ with a twist with a sufficiently relatively ample invertible sheaf.
Then we may assume that
\begin{equation*}
\pi^*\pi_* \F \to \F,\quad \pi^*\pi_*\G \to \G
\end{equation*}
are surjective (i.e.~$\F$ and $\G$ are globally generated over $S$) and, using flatness of $\G$, that $\pi_* \G$ is
locally free.
The zero locus of
\begin{equation*}
\pi_*(\phi)\colon \pi_*\F\to\pi_*\G,
\end{equation*}
in the usual sense, is the required scheme $Z_{\pi}(\phi)$.
It is closed by Nakayama.
\end{proof}

Next we compare with the theta-dual as defined by Pareschi--Popa \cite{PPminimal},
which we recall: let $D(-)$ denote the dualization functor $\HHom(-,\O)$.
Then, working in the derived category, the theta-dual is defined in \emph{loc.~cit.} as
the support of the $g$'th cohomology sheaf of $(-1)^*R\mathcal{S} RD(\I_Y(\Theta))$.
We have
\begin{align*}
(-1)^*R\mathcal{S} RD(\I_Y(\Theta)) &= (-1)^*Rp_{2*}(p_1^*RD(\I_Y(\Theta)) \tensor \P) \\
&\iso Rp_{2*} (1\times(-1))^*(p_1^*RD(\I_Y(\Theta))\tensor\P)\\
&\iso Rp_{2*} (p_1^*RD(\I_Y(\Theta))\tensor\P^{-1})\\
&\iso Rp_{2*} ( RD(p_1^*(\I_Y(\Theta)) \tensor \O(-m^*\Theta+p_1^*\Theta+p_2^*\Theta))) \\
&\iso \left(Rp_{2*} \left( RD(p_1^*(\I_Y(\Theta))\tensor\O(m^*\Theta-p_1^*\Theta))\right)\right) \tensor \O(\Theta)\\
&\iso \left(R(p_{2*}\circ D) (\I_{Y\times A}(m^*\Theta)\right) \tensor \O(\Theta).
\end{align*}
(In the third line we used the identity $(1\times(-1))^*\P \iso \P^{-1}$, in the fourth line
we used $\P = \O(m^*\Theta-p_1^*\Theta-p_2^*\Theta)$ and in the fifth line we used the
projection formula.)
The twist by $\Theta$ clearly does not affect the supports of the cohomologies of this complex,
and the $i$'th derived functor of $p_{2*}\circ D$ is the relative Ext-sheaf $\EExt^i_{p_2}(-,\O_{A\times A})$.
Thus we find that the theta-dual according to \emph{loc.~cit.} is the support of
$\EExt^g_{p_2}(\I_{Y\times A}(m^*\Theta), \O_{A\times A})$.
The following thus shows that our theta-dual agrees with the theta-dual of Pareschi--Popa:

\begin{prop}\label{prop:PPthetadual}
The sheaf
$\EExt^g_{p_2}(\I_{Y\times A}(m^*\Theta), \O_{A\times A})$
is isomorphic to the structure sheaf $\O_{T(Y)}$ of the theta-dual.
\end{prop}

\begin{proof}
We view $T(Y)$ as the zero locus $Z_{p_2}(\vartheta)$ of the canonical section $\vartheta$ of $\O_{Y\times
A}(m^*\Theta)$, relative to second projection $p_2\colon A\times A\to A$.
This agrees with Proposition \ref{prop:thetadual},
where $T(Y)$ was constructed as the relative zero locus for the inclusion $\I \subset \O_{Y\times A}$ of the ideal
$\I\iso \O_{Y\times A}(-m^*\Theta)$,
corresponding to the subscheme $m^{-1}(\Theta)\cap (Y\times A)$ of $Y\times A$.

Let us temporarily denote the support of $\EExt^g_{p_2}(\I_{Y\times A}(m^*\Theta), \O_{A\times A})$ by $T'(Y)$.
Apply relative Ext with respect to the second projection $p_2\colon A\times A\to A$ to the short exact sequence
\begin{equation*}
0 \to \I_{Y\times A}(m^*\Theta) \to \O_{A\times A}(m^*\Theta) \to \O_{Y\times A}(m^*\Theta) \to 0
\end{equation*}
to obtain the right exact
\begin{equation*}
\begin{split}
\EExt^g_{p_2}(\O_{Y\times A}(m^*\Theta), \O_{A\times A}) \to
&\EExt^g_{p_2}(\O_{A\times A}(m^*\Theta), \O_{A\times A})\\
\to &\EExt^g_{p_2}(\I_{Y\times A}(m^*\Theta), \O_{A\times A}) \to 0.
\end{split}
\end{equation*}
By relative duality, the sheaf in the middle is dual to $p_{2*}\O_{A\times A}(m^*\Theta)$,
which is $H^0(A, \O(\Theta))\tensor_k\O_A \iso \O_A$.
Thus $\EExt^g_{p_2}(\I_{Y\times A}(m^*\Theta), \O_{A\times A})$
is the structure sheaf of $T'(Y)$,
and the leftmost homomorphism has $T'(Y)$ as its vanishing locus.
Now let $D\subset A$ be an effective divisor not containing $Y$.
Then $T'(Y)$ is also the vanishing locus of the composite
\begin{equation*}
\begin{split}
\EExt^g_{p_2}(\O_{Y\times A}(m^*\Theta+p_1^*D), \O_{A\times A}) \onto
&\EExt^g_{p_2}(\O_{Y\times A}(m^*\Theta), \O_{A\times A}) \\
\to &\EExt^g_{p_2}(\O_{A\times A}(m^*\Theta), \O_{A\times A}).
\end{split}
\end{equation*}
If $D$ is sufficiently ample, we may apply relative duality to see
that this composite map is a homomorphism between locally free sheaves, dual to
\begin{equation*}
F\colon p_{2*}\O_{A\times A}(m^*\Theta) \to p_{2*}\O_{Y\times A}(m^*\Theta + p_1^*D).
\end{equation*}
Thus $T'(Y)$ is the vanishing locus of $F$.

The domain of $F$ is isomorphic with $p_{2*}\O_{A\times A}$.
Viewing $F$ as the map
\begin{equation*}
p_{2*}\O_{A\times A} \iso p_{2*}\O_{A\times A}(m^*\Theta) \to p_{2*}\O_{Y\times A}(m^*\Theta + p_1^*D),
\end{equation*}
we find that $T'(Y)$ is, in the language of Lemma \ref{lem:zero},
the zero locus relative to $p_2$ of a section in $H^0(\O_{Y\times A}(m^*\Theta+p_1^*D))$,
which is the product of two sections $\vartheta\in H^0(\O_{Y\times A}(m^*\Theta))$
and $d\in H^0(\O_{Y\times A}(p_1^*D))$.
Now $d$ has been chosen to be nonzero in all fibres of $p_2$,
i.e.~$Z_{p_2}(d)=\emptyset$.
It follows that $T(Y) = Z_{p_2}(\vartheta)$ and $T'(Y) = Z_{p_2}(d\cdot\vartheta)$ coincide.
\end{proof}

\begin{rem}
It is obvious from the universal property of the theta-dual that,
as subschemes of $A$,
we have $Y\subseteq T(T(Y))$ and $T(Y)\subseteq T(Y')$ whenever $Y'\subseteq Y$.
\end{rem}

\begin{ex}\label{ex:T(S)}
Let $S\subset A$ be a nonreduced degree two subscheme supported in a closed point $a$.
Then ``translation along $S$'' defines an infinitesimal deformation of $\Theta_a$,
namely the scheme $m^{-1}\Theta\cap(S\times A)$ as a family over $S$.
The theta-dual $T(S)\subset \Theta_a$ is the vanishing locus of the corresponding section
(defined up to scale, corresponding to the choice of an isomorphism $\Spec k[\epsilon]/(\epsilon^2)\iso S$) 
of the normal bundle $\O_{\Theta_a}(\Theta_a)$.
\end{ex}

\begin{ex}\label{ex:T(W)}
Let $J(C)$ be the Jacobian of a nonsingular projective curve $C$ of genus $g$,
and choose a base point of $C$.
Let $W_i$ be the image of the canonical map $C^{(i)}\to J(C)$ (defined using the base point).
We take $\Theta = W_{g-1}$ as the polarization.
Then it is clear that $W_i \times W_{g-i-1}$ is mapped to $\Theta$ under the group law,
so $W_{g-i-1} \subseteq T(W_i)$.
This inclusion is in fact an equality, as is shown in \cite[Section 8.1]{PPminimal}
(in the reference there is a sign change, which can be traced to the choice of identification
between $J(C)$ and its dual).
\end{ex}

\begin{ex}\label{ex:TS_Jac}
Let $J(C)$ be the Jacobian of a nonsingular projective curve $C$ of genus $g$.
For any two distinct points $p,q$ in $C\hookrightarrow \Pic^1(C)$,
we have the following equality \cite[Lecture IV]{mumford} of sets in $\Pic^{g-1}(C)$
\begin{equation*}
W_{g-1}\cap (W_{g-1})_{q-p}=(W_{g-2})_{-p}\cup (-W_{g-2})_{q-K},
\end{equation*}
where $K$ is a canonical divisor of $C$.
The translations are to be understood inside $\Pic(C)$, so that, e.g.~$(W_{g-2})_{-p} = W_{g-2}+p$ is a subset of
$\Pic^{g-1}(C)$.
For fixed $q$ and generic $p$, both sides of the equality are reduced in any case,
and form a flat family over an open subset of $C$.
Taking flat limits we extend the family to all of $C$.

We conclude that, when $S$ is a subscheme of degree 2 of an Abel-Jacobi curve $C\subset A$,
supported in two possibly coinciding points $a$ and $b$,
there is a schematic equality
\begin{equation*}
T(S) = (W_{g-2})_\alpha \cup (-W_{g-2})_{\beta},
\end{equation*}
with $\alpha$ and $\beta$ depending linearly on $a$ and $b$,
and if the two $\pm W_{g-2}$-translates on the right coincide,
their union is to be understood by perturbing $b$ and taking the flat limit.
Thus $T(S)$ is either the union of two distinct $\pm W_{g-2}$-translates,
or a scheme structure of multiplicity two on a single $W_{g-2}$-translate.
Moreover, the latter happens only in the hyperelliptic case:
it follows from example \ref{ex:T(W)} that $W_{g-2}$ and $-W_{g-2}$ coincide up to translation if and only if $C$ and
$-C$ do, which is equivalent to $C$ being hyperelliptic.
\end{ex}

\begin{ex}
Let $\Gamma\subset A$ be a finite subscheme.
In the short exact sequence
\begin{equation*}
0 \to \I_\Gamma(\Theta) \to \O_A(\Theta) \to \O_\Gamma \to 0,
\end{equation*}
the sheaf $\I_\Gamma(\Theta)$ satisfies WIT$_1$,
whereas the other two sheaves are IT$_0$.
Thus the Fourier-Mukai functor gives a short exact sequence
\begin{equation}\label{eq:theta-dual}
0\to \O_A(-\Theta) \xrightarrow{F} \widehat{\O_\Gamma} \to
\widehat{\I_\Gamma(\Theta)} \to 0.
\end{equation}
Choosing $D=0$ in the proof of Proposition \ref{prop:thetadual},
which is indeed sufficiently ample on the finite scheme $\Gamma$,
we find that the maps named $F$ in that proof and in \eqref{eq:theta-dual}
coincide up to twist by $\Theta$.
Thus the theta-dual $T(\Gamma)$ is precisely the zero locus of $F$ in \eqref{eq:theta-dual}.

Note that the fibre of $F$ over a point $a\in A$ is
\begin{equation*}
F(a)\colon H^0(\O_A(\Theta_a)) \to H^0(\O_{\Gamma}),
\end{equation*}
which vanishes precisely when $\Gamma\subset \Theta_a$.
Thus the zero locus of $F$ is, from the outset,
a natural scheme structure on the set of such points $a$.
The definition of Pareschi--Popa can be seen as a generalization of this observation,
where the lack of base change for $\widehat{\O_Y(\Theta)}$ is handled by working with the dual of $F$ instead.
\end{ex}

\begin{deff}\label{def:T-relative}
If $Y'\subset Y$ is a pair of subschemes of $A$,
we let $T(Y', Y)$ denote the schematic closure of $T(Y')\setminus T(Y)$ in $T(Y')$.
\end{deff}

Next we define theta-genericity:
recall that a finite subscheme $\Gamma$ in $\PP^r$ is in linearly
general position if every subscheme $\Gamma'\subseteq\Gamma$ of
degree $d\le r+1$ spans a linear subspace of dimension $d-1$.
Equivalently, for any pair of $\Gamma''\subset\Gamma'$ of
subschemes of $\Gamma$, such that
\begin{equation*}
\deg \Gamma' - 1 = \deg \Gamma'' \le r,
\end{equation*}
there exists a hyperplane containing $\Gamma''$ but not
$\Gamma'$. Phrased in this way, the condition of linear
independence can be carried over to $(A,\Theta)$, with
$\Theta$-translates replacing hyperplanes.

\begin{deff}[analogue of Definition 3.2 in \cite{PPschottky}]
A finite subscheme $\Gamma$ is \emph{theta-general} if, for every
pair $\Gamma''\subset\Gamma'$ of subschemes of $\Gamma$
satisfying
\begin{equation*}
\deg \Gamma' - 1 = \deg\Gamma'' \le g,
\end{equation*}
there exists a $\Theta$-translate containing $\Gamma''$ but not
$\Gamma'$.
\end{deff}

\begin{rem}[correction to Remark 3.5 in \cite{PPschottky}]
The condition in the definition demands that the inclusion
$T(\Gamma')\subset T(\Gamma'')$ is \emph{set theoretically} strict,
and not just scheme theoretically.
As an example, consider the surface case $g=2$.
There exist distinct points $a\ne b$ in $A$ such that $T(\set{a,b}) = \Theta_a\cap\Theta_b$
is a degree $2$ subscheme supported in a single point $x$.
Thus there exists a \emph{unique} theta-translate $\Theta_x$ containing $\set{a,b}$.
Let $c$ be a third point on $\Theta_x$, then $\set{a,b,c}$ is not theta-general:
the inclusion $T(\set{a,b,c})\subset T(\set{a,b})$ of schemes is strict,
but it is an equality of sets.
This might suggest that it is more natural to work with the weaker notion of theta-generality
given by demanding that $T(\Gamma')\subset T(\Gamma'')$ is a strict inclusion of schemes.
We will however continue to use the stronger, set theoretic, notion here.
\end{rem}

\subsection{Dependence loci}

Let $D\subset A$ be an ample divisor. In later sections, $D$ will
be taken to be $2\Theta$. We are concerned with the \emph{number
of independent conditions} imposed by a finite subscheme $\Gamma$ on the
linear system $D$. By this we mean the
codimension of $H^0(A,\I_\Gamma(D))$ in $H^0(A,\O_A(D))$.  As long as
$\deg \Gamma \le \dim H^0(A, \O_A(D))$, the
expected number of conditions imposed is the degree of $\Gamma$.
Since $D$ is ample, its higher cohomology spaces vanish, so there is
an exact sequence
\begin{equation*}
0 \to H^0(A, \I_\Gamma(D))
  \to H^0(A, \O_A(D))
  \to H^0(A, \O_\Gamma)
  \to H^1(A, \I_\Gamma(D))
  \to 0
\end{equation*}
which shows that $H^1(A, \I_\Gamma(D))$ measures the
failure of $\Gamma$ to impose $\deg \Gamma$ independent conditions.

We will in fact study the number of independent conditions imposed by
$\Gamma$ on all the linear systems associated to $H^0(A,
\O_A(D)\otimes\P_a)$ for $a\in A$. Since $D$ is ample,
the collection of these linear systems coincides with the collection of the translated systems $|D_a|$.

\begin{deff}
The \emph{superabundance} of a finite subscheme $\Gamma\subset A$ with
respect to $D$ is the value
\begin{equation*}
\omega(\Gamma, D) = \dim H^1(A, \I_\Gamma(D)\otimes\P_a)
\end{equation*}
for general $a\in A$. Equivalently, it is the minimal
value of the right hand side, over $a\in A$. The
subscheme $\Gamma$ is \emph{superabundant} if its superabundance is
nonzero.
\end{deff}

\begin{rem}
We deviate slightly from the literature (e.g.~Griffiths--Harris \cite{GH}), where superabundance
$\omega(\Gamma, D)$ is defined without the twist by a generic $\P_a$, but otherwise as above.
\end{rem}

It is also useful to study the locus of points
$a\in A$ such that $\Gamma$ does not impose
independent conditions on $H^0(A,\O_A(D)\otimes\P_a)$.

\begin{deff}
The \emph{dependence locus} $\Delta(\Gamma,D)$ is the Fitting support of
\begin{equation*}
R^1\mathcal{S}(\I_\Gamma(D)).
\end{equation*}
\end{deff}

\begin{rem}
Note that $R^i\mathcal{S}(\I_\Gamma(D))$ vanish for
all $i>1$. Hence, by base change, the fibre of
$R^1\mathcal{S}(\I_\Gamma(D))$ at $a$ is
\begin{equation*}
H^1(A, \I_\Gamma(D)\otimes\P_a)
\end{equation*}
which is nonzero precisely when $\Gamma$ fails to impose
independent conditions on the linear system associated to $\O_A(D)\otimes\P_a$.
\end{rem}

Let $\Gamma'\subset \Gamma$ be a pair of finite subschemes.
There is an exact sequence
\begin{equation*}
0 \to \I_{\Gamma}(D)
\to \I_{\Gamma'}(D)
\to \I_{\Gamma'/\Gamma}
\to 0.
\end{equation*}
Applying the Fourier-Mukai functor, we obtain a right exact sequence
\begin{equation}\label{eq:phi-sequence}
\widehat{\I_{\Gamma'/\Gamma}}
\to  R^1\mathcal{S}(\I_{\Gamma}(D))
\xrightarrow{\phi} R^1\mathcal{S}(\I_{\Gamma'}(D)) \to 0.
\end{equation}

\begin{deff}
Given a pair $\Gamma'\subset\Gamma$ of finite
subschemes, their \emph{relative dependence locus} is
the Fitting support
\begin{equation*}
\Delta(\Gamma',\Gamma,D) = \Fitt(\Ker (\phi)),
\end{equation*}
where $\phi$ is the map in \eqref{eq:phi-sequence}.
\end{deff}

\begin{rem}
The (underlying set of the) dependence locus $\Delta(\Gamma, D)$ is denoted $V(\I_{\Gamma}(D))$ in \cite[Definition
3.9]{PPschottky}.
Our relative dependence locus $\Delta(\Gamma',\Gamma)$ plays a similar role to the set $B(\I_{\Gamma'}(D),p)$ in
\cite[Definition 3.10]{PPschottky}, when $\Gamma = \Gamma' \cup \set{p}$, although they are not identical.
\end{rem}

\begin{lem}\label{lem:triple}
There are inclusions
\begin{equation*}
\Delta(\Gamma',D)\subseteq
\Delta(\Gamma,D) \subseteq
\Delta(\Gamma',D) \cup
\Delta(\Gamma',\Gamma,D)
\end{equation*}
where the union is defined by the \emph{product} of the corresponding ideals.
\end{lem}

\begin{proof}
This follows from \eqref{eq:phi-sequence}.
\end{proof}

%
%

\subsection{Residual subschemes}

Following Eisenbud--Green--Harris \cite{EGH},
we define a scheme theoretic version of ``complement'' as follows:

\begin{deff}
Let $\Gamma$ be a finite scheme and $\Gamma'\subset\Gamma$ a
subscheme. The \emph{residual subscheme} of $\Gamma'$ in $\Gamma$
is the support
\begin{equation*}
\Gamma'' = \supp \I_{\Gamma'/\Gamma}
\end{equation*}
of the ideal of $\Gamma'$ in $\Gamma$. If
$\I_{\Gamma'/\Gamma}$ is a principal ideal,
then we say that $\Gamma''$ is \emph{well formed}.
\end{deff}

\begin{rem}\label{rem:wf}
When the residual subscheme is well formed, we may
identify $\I_{\Gamma'/\Gamma}$ with the
structure sheaf $\O_{\Gamma''}$, so there is
a short exact sequence
\begin{equation*}
0 \to \O_{\Gamma''} \to \O_{\Gamma} \to
\O_{\Gamma'} \to 0.
\end{equation*}
In particular, there is an equality
\begin{equation}\label{eq:residualcycles}
[\Gamma] = [\Gamma'] + [\Gamma'']
\end{equation}
between the underlying zero-cycles.
\end{rem}

\begin{rem}\label{rem:resunion}
Let the union $\Gamma'\cup\Gamma''$ denote the
subscheme, inside some ambient scheme, defined by the product of the
corresponding ideals. Then it is immediate from the
definition of the residual subscheme (not
necessarily well formed) that
\begin{equation*}
\Gamma \subset \Gamma'\cup\Gamma''.
\end{equation*}
In particular, if $D$ is an effective divisor containing
$\Gamma'$, then there is an inclusion of ideals
\begin{equation*}
\I_{\Gamma''}(-D) \subset \I_\Gamma.
\end{equation*}
\end{rem}

\begin{ex}\label{ex:respoint}
If $\Gamma'$ has degree $\deg \Gamma - 1$, then the ideal
$\I_{\Gamma'/\Gamma}$ is isomorphic to the residue field
$k(x)$ of the unique closed point $x$ where $\Gamma'$ and $\Gamma$
differ. Thus $x$ is the \emph{residual point} of $\Gamma'$ in
$\Gamma$, and it is well formed.
\end{ex}

\begin{deff}[Le Barz \cite{lebarz}]
A finite subscheme $Z$ of a variety $X$ is \emph{curvilinear}
if there is a reduced curve $C\subset X$ containing $Z$,
and such that $C$ is smooth along the support of $Z$.
\end{deff}

\begin{ex}\label{ex:rescurvi}
If $\Gamma$ is curvilinear, and $\Gamma'\subset\Gamma$ is
arbitrary, then the residual subscheme $\Gamma''$ is well formed. In
fact, it is uniquely determined by \eqref{eq:residualcycles}.
\end{ex}

\begin{ex}
If $\Gamma = \Spec k[x,y]/(x^2,xy,y^2)$ and $\Gamma'$ is the
origin, then it is clear that the data given does not
distinguish any degree $2$ subscheme of $\Gamma$. Indeed, the
residual subscheme is just the origin again, so it is not well
formed.
\end{ex}

\begin{lem}\label{lem:resgor}
If $\Gamma$ is Gorenstein, and $\Gamma'\subset\Gamma$ has degree
$\deg \Gamma - 2$, then the residual subscheme
$\Gamma''$ of $\Gamma'$ in $\Gamma$ is well formed.
\end{lem}

\begin{proof}
If $\Gamma'$ and $\Gamma$ differ at two distinct points $x$ and
$y$, then the ideal of $\Gamma'$ in $\Gamma$ is just $k(x)\oplus
k(y)$ and thus $\Gamma'' = \set{x,y}$.

On the
other hand, if $\Gamma''$ and $\Gamma$ differ at a single point $x$,
then locally at $x$, we have
\begin{equation*}
\Gamma = \Spec(R)
\end{equation*}
for a local Artin Gorenstein ring $R$. The ideal of $\Gamma''$ in $R$
is two dimensional as a vector space. Hence it is either a
principal ideal, or it is generated by two linearly independent elements
of the socle of $R$. The Gorenstein assumption rules out the latter
possibility, so the ideal of $\Gamma''$ in $R$ is
principal.
\end{proof}

\section{Superabundance and dependence loci}\label{sec:key}

From here on, we fix $D=2\Theta$ and abbreviate
$\Delta(\Gamma,2\Theta)$, $\Delta(\Gamma',\Gamma, 2\Theta)$ and
$\omega(\Gamma,2\Theta)$ to
$\Delta(\Gamma)$, $\Delta(\Gamma',\Gamma)$ and
$\omega(\Gamma)$.

\begin{lem}\label{lem:deltatheta}
Let $\Gamma'\subset\Gamma$ be finite subschemes of $A$ of
such that $\deg\Gamma' = \deg\Gamma-1$, and let $a$
denote the residual point. Then we have an inclusion of
schemes
\begin{equation*}
\Delta(\Gamma',\Gamma) \subseteq \Theta_{a-y}
\end{equation*}
for all closed points $y\in T(\Gamma')\setminus T(\Gamma)$.
\end{lem}

\begin{proof}
Since $\Theta_y$ contains $\Gamma'$, but not $\Gamma$, we
have a commutative diagram
\begin{equation*}
\xymatrix{
 0 \ar[r] &  \I_{\Gamma}          \ar[r]  & \I_{\Gamma'} \ar[r] & k(a) \ar[r] & 0\\
 0 \ar[r] &  \I_{\set{a}}(-\Theta_y)\ar@{^{(}->}[u] \ar[r] & \O(-\Theta_y) \ar@{^{(}->}[u] \ar[r] & k(a) \ar@{=}[u]
\ar[r] & 0
}
\end{equation*}
with exact rows. Twist with $2\Theta$, use that
$2\Theta-\Theta_y$ is linearly equivalent to $\Theta_{-y}$, and
apply the Fourier-Mukai transform to arrive at the commutative diagram
\begin{equation*}
\xymatrix{
 \P_a \ar@{=}[d]\ar[r] & R^1\mathcal{S}(\I_{\Gamma}(2\Theta)) \ar[r]^{\phi} & R^1\mathcal{S}(\I_{\Gamma'}(2\Theta))
\ar[r]& 0 \\
 \P_a \ar[r]      & \F \ar[u]\ar[r] & 0
}
\end{equation*}
with exact rows, and where $\F$ is the Fourier-Mukai transform
of the WIT$_1$ sheaf $\I_{\set{a}}(\Theta_{-y})$. A small calculation
shows that $\F\cong \res{\P_a}{\Theta_{a-y}}$. By definition,
$\Delta(\Gamma',\Gamma)$ is the support of
$\Ker(\phi)$. Since the kernel of $\phi$ is a quotient of $\F$,
it follows that its support is contained in the support of
$\F$, which gives the claim.
\end{proof}

\begin{lem}[analogue of Lemma 3.13 in \cite{PPschottky}]\label{lem:codim2}
Let $\Gamma$ be a theta-general finite subscheme of $A$ of degree at most $g$.
Then $\Delta(\Gamma)$ has codimension at least $2$.
\end{lem}

\begin{proof}
Induct on the degree $d$ of $\Gamma$:
let $\Gamma'\subset\Gamma$ be a subscheme of degree $d-1$.
Its theta-dual $T(\Gamma')$ is locally defined by $d-1$ equations,
hence has positive dimension everywhere.
The inclusion $T(\Gamma) \subset T(\Gamma')$ is strict by theta-genericity,
so $T(\Gamma')\setminus T(\Gamma)$ has positive dimension everywhere.
By Lemma \ref{lem:deltatheta}, it follows that $\Delta(\Gamma',\Gamma)$ has
codimension at least $2$.
The inclusion $\Delta(\Gamma) \subseteq \Delta(\Gamma')\cup\Delta(\Gamma',\Gamma)$
from Lemma \ref{lem:triple} concludes the induction.
\end{proof}

\subsection{Superabundant subschemes}\label{sec:super}

It is to be expected that the superabundance $\omega(\Gamma)$ (always with
respect to $2\Theta$) vanishes as long as $\Gamma$ has small degree.
We begin by establishing that the minimal degree of a superabundant
theta-general subscheme is $g+2$.

\begin{prop}
Let $\Gamma\subset A$ be a theta-general finite subscheme of degree
at most $g+1$. Then $\omega(\Gamma)=0$.
\end{prop}

\begin{proof}
The claim is that $\Delta(\Gamma)$ has codimension at least one.
As in Lemma \ref{lem:codim2}, induct on the degree $d$ of $\Gamma$:
let $\Gamma'\subset\Gamma$ be a degree $d-1$ subscheme. 
Then $T(\Gamma')\setminus T(\Gamma)$ is nonempty,
so $\Delta(\Gamma',\Gamma)$ has codimension at least one by Lemma \ref{lem:deltatheta}.
The inclusion $\Delta(\Gamma)\subseteq\Delta(\Gamma')\cup\Delta(\Gamma',\Gamma)$
concludes the induction.
\end{proof}

The above bound is sharp: on a Jacobian $A$ there exist
superabundant subschemes of degree $g+2$. In fact, by Riemann-Roch,
an Abel-Jacobi curve $C\subset A$ imposes $g+1$ independent
conditions on $H^0(A, \O_A(2\Theta)\otimes\P_x)$ for any
$x$. Hence a finite subscheme $\Gamma$ of $C$, no matter how
big, cannot impose more than $g+1$ conditions. See Pareschi--Popa
\cite[Example 3.7]{PPschottky} for a more precise statement. Our main
Theorem \ref{thm} says that subschemes of Abel-Jacobi curves are
the only (theta-general)
examples of superabundant subschemes of degree $g+2$.

\begin{cor}\label{cor:cayleybach}
Let $\Gamma\subset A$ be a theta-general, superabundant finite
subscheme of degree $g+2$, and let $\Gamma'\subset\Gamma$ have degree
$g+1$. Then any theta-translate containing $\Gamma'$ also contains
$\Gamma$, i.e.~$T(\Gamma',\Gamma)=\emptyset$.
\end{cor}

\begin{proof}
By the same argument as in the proof of the proposition, the
existence of a point in $T(\Gamma',\Gamma)$ would imply that
$\Delta(\Gamma)\ne A$, hence $\Gamma$ could not be superabundant.
\end{proof}

\begin{cor}\label{cor:gor}
Let $\Gamma\subset A$ be a theta-general, superabundant finite
subscheme of degree $g+2$. Then $\Gamma$ is Gorenstein, i.e.~each
component of $\Gamma$ is the spectrum of a Gorenstein ring.
\end{cor}

\begin{proof}
Let $\Gamma_0\subset \Gamma$ be a component, so $\Gamma_0=\Spec A$ for
a local Artin ring $A$. We need to show that the socle $\Soc(A)$,
i.e.~the elements in $A$ annihilated by its maximal ideal, is one dimensional
as a vector space. For contradiction, assume $f,g\in \Soc(A)$ are
linearly independent elements. Since the ideal generated by
any collection of socle elements coincides with the vector space they span,
the ideals $(f,g)$ and $(f)$ in $A$ determine subschemes
\begin{equation*}
\Gamma''\subset\Gamma'
\end{equation*}
in $\Gamma$, of degree $g$ and $g+1$, respectively (precisely, $\Gamma'$
is the union of $\Gamma\setminus\Gamma_0$ with the subscheme of
$\Gamma_0$ defined by $(f)$, and similarly for $\Gamma''$). By
theta-genericity, there exists a theta-translate $\Theta_a$ containing
$\Gamma''$ but not $\Gamma'$. Let $\vartheta\in A$ be a local equation
for $\Theta_a$. Then $\vartheta$ is a socle element, since $\vartheta\in(f,g)$,
and hence defines a subscheme $Z\subset\Gamma$ of degree $g+1$. But then
$Z$ is contained in $\Theta_a$, and $\Gamma$ is not, contradicting
the previous corollary.
\end{proof}

\subsection{The key lemma}

\begin{lem}[analogue of Lemma 5.1 in \cite{PPschottky}]\label{lem:key}
Let $\Gamma_g\subset\Gamma_{g+1}\subset\Gamma_{g+2}$ be finite subschemes
of $A$ of degrees indicated by the subscripts,
and assume $\Gamma_{g+2}$ is theta-general and superabundant.
Then the following hold.
\begin{enumerate}
\item[\emph{(i)}] There exists a unique theta-translate $\Theta_x$ containing
$\Gamma_g$ but not $\Gamma_{g+2}$.
\item[\emph{(ii)}] The divisorial part of $\Delta(\Gamma_{g+1})$ is reduced and
equals $\Theta_{b-x}$, where $x$ is as above and $b$ is the residual point of 
$\Gamma_g\subset\Gamma_{g+1}$.
\end{enumerate}
\end{lem}

\begin{rem}\label{rem:translates}
By Corollary \ref{cor:cayleybach}, the theta-translate
$\Theta_x$ in part (i) cannot contain $\Gamma_{g+1}$. Thus $\Theta_x$
is also the unique theta-translate containing $\Gamma_{g}$ but not
$\Gamma_{g+1}$.
\end{rem}

Before proving the Lemma, we explain a consequence that will be important
for proving the Castelnuovo part of Theorem \ref{thm}.

\begin{cor}\label{cor:curv}
Let $\Gamma\subset A$ be a theta-general, superabundant finite
subscheme of degree $g+2$. Then $\Gamma$ is curvilinear.
\end{cor}

\begin{proof}
Since $\Gamma$ is Gorenstein, it suffices to show that every
subscheme $\Gamma_{g+1}\subset\Gamma$ of degree $g+1$ is also
Gorenstein. In other words, Gorenstein means the choice of a point
$a\in\Gamma$ uniquely determines a subscheme $\Gamma_{g+1}\subset\Gamma$
with residual point $a$. If also the choice of a residual point
$b\in\Gamma_{g+1}$ uniquely determines $\Gamma_{g}\subset\Gamma_{g+1}$, then
$\Gamma$ is curvilinear \cite[Lemma 1.4]{EH}.

Thus we suppose that $\Gamma_{g}^1$ and $\Gamma_{g}^2$ are two
subschemes of $\Gamma_{g+1}$ of degree $g$ with residual point $b$. By
the first part of the Lemma, there are unique points $x_1$ and $x_2$
such that $\Gamma_{g}^i$ is contained in $\Theta_{x_i}$, but $\Gamma_{g+1}$
is not. By the second part of the lemma, the divisorial part of $\Delta(\Gamma_{g+1})$ equals $\Theta_{b-x_i}$, for
either $i$,
and so $x_1=x_2$. Call this point $x$. Then $\Theta_x$ contains both
$\Gamma_{g}^1$ and $\Gamma_{g}^2$, but not $\Gamma_{g+1}$, which is impossible
unless $\Gamma_{g}^1 = \Gamma_{g}^2$.
\end{proof}

\begin{proof}[Proof of Lemma \ref{lem:key}]
By theta-genericity, there exists a theta-translate $\Theta_x$ that
contains $\Gamma_g$, but not $\Gamma_{g+1}$ (or equivalently, not $\Gamma_{g+2}$, by
Remark \ref{rem:translates}).
We claim there are inclusions
\begin{equation}\label{eq:inclusions}
\Theta_{b-x} \subseteq \Delta(\Gamma_{g+1}) \subseteq
\Theta_{b-x}\cup \Delta(\Gamma_g),
\end{equation}
where the union on the right hand side is scheme theoretically defined by
taking the product of the corresponding ideals.
Since $\Delta(\Gamma_g)$ has codimension at least two,
by Lemma \ref{lem:codim2}, this immediately gives part (ii) of the Lemma.
Then also part (i) follows,
since $\Theta_{b-x}$ and thus $x$ is uniquely determined by the pair $\Gamma_g\subset\Gamma_{g+1}$.

Now we prove \eqref{eq:inclusions}.
The finite scheme $\Gamma_{g+2}$ is Gorenstein by Corollary \ref{cor:gor}.
Thus, by Lemma \ref{lem:resgor},
the residual scheme $S$ of $\Gamma_g$ in $\Gamma_{g+2}$ is well formed,
so it has degree two.
Let $a$ be the residual point of $\Gamma_{g+1}$ in $\Gamma_{g+2}$.
By Corollary \ref{cor:cayleybach}, the intersection $\Gamma_{g+2}\cap \Theta_x$
cannot have degree $g+1$, so it must equal $\Gamma_g$.
In other words, locally, the ideal of $\Gamma_g$ is generated by the ideal of $\Gamma_{g+2}$ together with a local
equation for $\Theta_x$.
From this one deduces that there is an exact commutative diagram:
\begin{equation*}
\xymatrix{
&0&0\\
& \I_{\Gamma_g/\Theta_x} \ar@{=}[r]\ar[u] & \I_{\Gamma_g/\Theta_x}\ar[u]\\
0\ar[r] & \I_{\Gamma_{g+2}} \ar[u]\ar[r] & \I_{\Gamma_{g+1}} \ar[r]\ar[u] & k(a)  \ar[r]&0 \\
0\ar[r] & \I_S(-\Theta_x) \ar[u]\ar[r] & \I_{\{b\}}(-\Theta_x) \ar[u]\ar[r] & k(a) \ar@{=}[u] \ar[r]&0 \\
&0\ar[u]&0\ar[u]
} 
\end{equation*}
Now twist this diagram by $2\Theta$, note that
$2\Theta-\Theta_x$ is linearly equivalent to $\Theta_{-x}$, and apply
the Fourier-Mukai transform. Using the short exact sequence \eqref{eq:theta-dual}
with $Z=\set{b}$, we find
\begin{equation*}
\widehat{\I_{\{b\}}(\Theta_{-x})} \iso \res{\P_b}{\Theta_{b-x}}.
\end{equation*}
With $Z=S$, the vanishing locus $T(S)$ of $F$ in \eqref{eq:theta-dual} has
codimension two, hence the cokernel $\widehat{\I_S(\Theta_{-x})}$ is torsion free of rank one,
i.e.~it is a twist of the ideal of $T(S)$. It follows that
\begin{equation*}
\widehat{\I_S(\Theta_{-x})} \iso \I_{T(S)_{-x}}(\Theta_{a+b-x})
\end{equation*}
and so we arrive at the exact commutative diagram:
\begin{equation}\label{eq:keydiag}
\xymatrix{
&&0&0\\
&& R^1\mathcal{S}(\I_{\Gamma_g/\Theta_x}(2\Theta))\ar@{=}[r]\ar[u]& R^1\mathcal{S}(\I_{\Gamma_g/\Theta_x}(2\Theta))
\ar[u]\\
& \P_a \ar[r] & R^1\mathcal{S}(\I_{\Gamma_{g+2}}(2\Theta))\ar[u] \ar[r]^\rho &
R^1\mathcal{S}(\I_{\Gamma_{g+1}}(2\Theta)) \ar[u]\ar[r] & 0 \\
0 \ar[r] & \P_a \ar@{=}[u]\ar[r] & \I_{T(S)_{-x}}(\Theta_{a+b-x}) \ar[u]_\mu\ar[r] & \res{\P_b}{\Theta_{b-x}} 
\ar[u]_\nu \ar[r]&0 
}
\end{equation}

We claim that $\nu$ is injective.
First note that $\rho$ is not an isomorphism,
since its domain has $\Delta(\Gamma_{g+2})=A$ as support ($\Gamma_{g+2}$ is superabundant), whereas its codomain is
torsion, supported in $\Delta(\Gamma_{g+1})$.
It follows that $\mu$ is nonzero, and having torsion free domain of rank one, it is injective.
By a diagram chase
we find that also $\nu$ is injective.
The vertical short exact sequence on the right thus shows
\begin{equation*}
\Theta_{b-x} \subseteq \Delta(\Gamma_{g+1}) \subseteq
\Theta_{b-x}\cup \Fitt R^1\mathcal{S}(\I_{\Gamma_g/\Theta_x}(2\Theta)).
\end{equation*}
Thus we are done once we have shown that
\begin{equation*}
R^1\mathcal{S}(\I_{\Gamma_g}(2\Theta))\iso R^1\mathcal{S}(\I_{\Gamma_g/\Theta_x}(2\Theta)),
\end{equation*}
and in fact, such an isomorphism is obtained from the short exact sequence
\begin{equation*}
0 \to \O_A(-\Theta_x) \to \I_{\Gamma_g} \to \I_{\Gamma_g/\Theta_x} \to 0
\end{equation*}
by twisting with $2\Theta$ and then applying the Fourier-Mukai transform.
\end{proof}

\subsection{Sums of theta-duals and dependence loci}

The following Lemma will be an essential ingredient in our proof of the Castelnuovo statement.

\begin{lem}\label{lem:Tsum}
Let $\Sigma_i\subset\Gamma_{g+1}\subset\Gamma_{g+2}$ be finite subschemes of $A$,
of degree indicated by their subscripts, and assume that $\Gamma_{g+2}$ is
theta-general and superabundant.
Define $x$ to be the residual point of $\Gamma_{g+1}$ in $\Gamma_{g+2}$,
let $j = g+1-i$ and let $\Sigma_{i+1}$ and $\Lambda_j\subset\Lambda_{j+1}$ be subschemes of $\Gamma_{g+2}$ such that the
underlying zero cycles satisfy
\begin{align*}
[\Gamma_{g+1}] &= [\Sigma_i] + [\Lambda_j] &
[\Sigma_{i+1}] &= [\Sigma_i] + x &
[\Lambda_{j+1}] &= [\Lambda_j] + x.
\end{align*}
Then there is an inclusion of schemes
\begin{equation*}
T(\Sigma_i,\Sigma_{i+1}) + T(\Lambda_j,\Lambda_{j+1})
\subseteq \Delta(\Gamma_{g+1})
\end{equation*}
where the left hand side denotes the scheme theoretic image of
$T(\Sigma_i,\Sigma_{i+1})\times T(\Lambda_j,\Lambda_{j+1})$ under the
group law $m \colon A\times A \to A$.
\end{lem}

\begin{proof}
Note that the equalities of zero cycles define the various finite subschemes
uniquely, as $\Gamma_{g+2}$ is curvilinear by Corollary \ref{cor:curv}.

We rephrase the statement a little: since formation of Fitting ideals
commute with base change, it suffices to show that
\begin{equation}\label{eq:T-inclusion}
T(\Sigma_i,\Sigma_{i+1})\times T(\Lambda_j,\Lambda_{j+1}) 
\subseteq
\Fitt(\nu^*R^1\mathcal{S}(\I_{\Gamma_{g+1}}(2\Theta))),
\end{equation}
where
\begin{equation*}
\nu\colon T(\Sigma_i)\times T(\Lambda_j) \to A
\end{equation*}
is the restriction of the group law.

To understand the right hand side of \eqref{eq:T-inclusion}, we begin
with the short exact sequence
\begin{equation}\label{eq:ses-x}
0 \to \I_{\Gamma_{g+2}}(2\Theta) \to \I_{\Gamma_{g+1}}(2\Theta) \to k(x) \to 0.
\end{equation}
Instead of first applying Fourier-Mukai, and then pulling back by $\nu$,
we encode both operations in the functor $\T$ sending a sheaf
$\F$ on $A$ to the sheaf
\begin{equation*}
\T(\F) = p_{23*}(p_1^*(\F)
\tensor (1\times \nu)^*\P)
\end{equation*}
on $T(\Sigma_i)\times T(\Lambda_j)$. Here $p_k$ and $p_{kl}$ denote the
various projections from $A\times T(\Sigma_i)\times T(\Lambda_j)$.
In standard terminology, $\T$ (or its total derived functor)
is the Fourier-Mukai transform with kernel
\begin{equation}\label{eq:Tkernel}
(1\times \nu)^*\P
\iso p_{12}^*(\res{\P}{A\times T(\Sigma_i)})\tensor p_{13}^*(\res{\P}{A\times T(\Lambda_j)}).
\end{equation}
Applying $\T$ to \eqref{eq:ses-x}, we get a long exact sequence
\begin{multline*}
0 \to \T(\I_{\Gamma_{g+2}}(2\Theta))
\to \T(\I_{\Gamma_{g+1}}(2\Theta))
\to \T(k(x))\\
\to R^1\T(\I_{\Gamma_{g+2}}(2\Theta))
\to R^1\T(\I_{\Gamma_{g+1}}(2\Theta))
\to 0.
\end{multline*}
If $\F$ is a sheaf on $A$ and $p$ is maximal such that $R^p\mathcal{S}(\F)\ne 0$,
then base change shows that $\nu^*R^p\mathcal{S}(\F)\iso R^p\T(\F)$. Using
this, we can rewrite the last few terms in the long exact sequence, and obtain
\begin{multline}\label{eq:phi}
0 \to \T(\I_{\Gamma{g+2}}(2\Theta))
\to \T(\I_{\Gamma_{g+1}}(2\Theta))
\xrightarrow{\phi} \nu^*\P_x\\
\to \nu^*R^1\mathcal{S}(\I_{\Gamma_{g+2}}(2\Theta))
\to \nu^*R^1\mathcal{S}(\I_{\Gamma_{g+1}}(2\Theta))
\to 0.
\end{multline}
As $\Gamma_{g+2}$ is superabundant, the support of
$R^1\mathcal{S}(\I_{\Gamma_{g+2}}(2\Theta))$ is all of $A$,
so the Fitting support of
its pullback by $\nu$ is all of $T(\Sigma_i)\times T(\Lambda_j)$. So it is enough to prove the following claim.

\begin{claim}\label{cl:phi}
The homomorphism $\phi$ in the long exact sequence \eqref{eq:phi} is surjective over
$(T(\Sigma_i)\setminus T(\Sigma_{i+1}))\times(T(\Lambda_j)\setminus
T(\Lambda_{j+1}))$.
\end{claim}

To prove the claim, we use the commutative diagram
\begin{equation*}
\xymatrix{
\I_{\Sigma_i}(\Theta)\otimes\I_{\Lambda_j}(\Theta)\ar@{>>}[d]\ar[r]&\I_{\Gamma_{g+1}}(2\Theta)\ar@{>>}[d]\\ 
k(x)\otimes k(x)\ar[r]^{\cong}& k(x),
}
\end{equation*}
where the vertical arrows are the evaluation maps at $x$ as in \eqref{eq:ses-x}, and the horizontal arrows
are multiplication maps.
The top map is well defined by Remark \ref{rem:resunion}.
If we apply the functor $\T$ to the previous diagram, we get
\begin{equation}\label{eq:diag2}
\xymatrix{
\T\left(\I_{\Sigma_i}(\Theta)\otimes\I_{\Lambda_j}(\Theta)\right)\ar[d]\ar[r]&\T\left(\I_{\Gamma_{g+1}}
(2\Theta)\right)\ar[d]\\ 
\T\left(k(x) \otimes k(x)\right)\ar[r]^{\cong}& \nu^*\P_x
}
\end{equation}
Using \eqref{eq:Tkernel}, we compute
\begin{equation}\label{eq:TIL}
\T(\I_{\Sigma_i}(\Theta)\otimes\I_{\Lambda_j}(\Theta))
\cong p_{23*}\left(\mathcal{G}_1\otimes \mathcal{G}_2\right)
\end{equation}
where $\mathcal{G}_1=p_{12}^*(p_{1}^*\I_{\Sigma_i}(\Theta)\otimes\res{\P}{A\times T(\Sigma_i)})$ and
$\mathcal{G}_2=p_{13}^*(p_{1}^*\I_{\Lambda_j}(\Theta)\otimes \res{\P}{A\times T(\Lambda_j)})$.

Consider now the natural map
\begin{equation}
p_{23*}(\mathcal{G}_1)\otimes p_{23*}(\mathcal{G}_2)\stackrel{\varrho}\to
p_{23*}(\mathcal{G}_1\otimes\mathcal{G}_2).\label{eq:varrho}
\end{equation}
By \eqref{eq:TIL} we recognize the codomain of $\varrho$.
To understand the domain, we define the functor $\T_1$, sending a sheaf $\F$ on $A$ to the sheaf
\begin{equation*}
\T_1(\F) = p_{2*}(p_1^*(\F)\otimes\res{\P}{A\times T(\Sigma_i)})
\end{equation*}
on $T(\Sigma_i)$.
In other words, $\T_1$ is the Fourier-Mukai transformation with kernel $\res{\P}{A\times T(\Sigma_i)}$.
Analogously, let $\T_2$ be the Fourier-Mukai transformation with kernel $\res{\P}{A\times T(\Lambda_j)}$.
With this notation, and by \eqref{eq:TIL}, the homomorphism $\varrho$ becomes
\begin{align*}
p_1^*\T_1(\I_{\Sigma_i}(\Theta))\otimes  p_2^*\T_2(\I_{\Lambda_j}(\Theta))\stackrel{\varrho}\to
\T(\I_{\Sigma_i}(\Theta)\otimes\I_{\Lambda_j}(\Theta)).
\end{align*}
Analogously we have a natural homomorphism
\begin{equation*}
p_1^*\T_1(k(x))\otimes  p_2^*\T_2(k(x))\stackrel{\varrho'}\to \T(k(x)\otimes k(x)).
\end{equation*}
Identifying its domain with $p_1^*(\res{\P_x}{T(\Sigma_i)})\otimes p_2^*(\res{\P_x}{T(\Lambda_j)})$,
and its codomain with $\nu^*\P_x$,
we see that this map is the restriction to $T(\Sigma_i)\times T(\Lambda_j)$ of the canonical isomorphism
$p_1^*\P_x\otimes p_2^*\P_x\to m^*\P_x$.

Therefore, composing the vertical maps in diagram \eqref{eq:diag2} with $\varrho$ and $\varrho'$, we get the following
commutative diagram of sheaves on $T(\Sigma_i)\times T(\Lambda_j)$:
\begin{equation}\label{eq:phi-diagram}
\xymatrix{
p_1^*\T_1(\I_{\Sigma_i}(2\Theta))\otimes p_2^*\T_2(\I_{\Lambda_j}(2\Theta))\ar[r]\ar[d] & \T(\I_{\Gamma_{g+1}}(2\Theta))
\ar[d]^{\phi} \\
p_1^*(\res{\P_x}{T(\Sigma_i)})\otimes p_2^*(\res{\P_x}{T(\Lambda_j)}) \ar[r]^(.6){\cong}& \nu^*\P_x. 
}
\end{equation}
Thus Claim \ref{cl:phi} follows if we can prove that the leftmost vertical map in this diagram surjects over
$(T(\Sigma_i)\setminus T(\Sigma_{i+1}))\times(T(\Lambda_j)\setminus T(\Lambda_{j+1}))$.
In fact:

\begin{claim}\label{cl:T1}
The map
$\T_1(\I_{\Sigma_i}(2\Theta))\to\res{\P_x}{T(\Sigma_i)}$
is a homomorphism between invertible sheaves on $T(\Sigma_i)$, with vanishing locus
$T(\Sigma_{i+1})$.
\end{claim}

This claim, together with the analogous statement for $\T_2$, shows that, over $(T(\Sigma_i)\setminus
T(\Sigma_{i+1}))\times(T(\Lambda_j)\setminus T(\Lambda_{j+1}))$,
 the leftmost vertical map in
\eqref{eq:phi-diagram} is a nonvanishing map between invertible sheaves. Hence it is an isomorphism, and in particular
it is surjective.
 
It remains to prove Claim \ref{cl:T1}.
Associated to $\Sigma_i\subset\Sigma_{i+1}$ there is a commutative diagram
\begin{equation*}
\xymatrix{
0 \ar[r] & \I_{\Sigma_i}(\Theta) \ar[r] \ar@{->>}[d] & \O(\Theta) \ar[r] \ar@{->>}[d] & \O_{\Sigma_i} \ar[r] \ar@{=}[d]
& 0 \\
0 \ar[r] & k(x) \ar[r] & \O_{\Sigma_{i+1}} \ar[r] & \O_{\Sigma_i} \ar[r] & 0
}
\end{equation*}
with exact rows. Apply $\T_1$ to obtain the commutative diagram
\begin{equation*}
\xymatrix{
0 \ar[r] & \T_1(\I_{\Sigma_i}(\Theta)) \ar[r] \ar[d] & \res{\widehat{\O(\Theta)}}{T(\Sigma_i)} \ar[r]^0 \ar[d] &
\res{\widehat{\O_{\Sigma_i}}}{T(\Sigma_i)} \ar@{=}[d] \\
0 \ar[r] & \res{\P_x}{T(\Sigma_i)} \ar[r] & \res{\widehat{\O_{\Sigma_{i+1}}}}{T(\Sigma_i)} \ar[r] &
\res{\widehat{\O_{\Sigma_i}}}{T(\Sigma_i)}
}
\end{equation*}
with exact rows.
The vanishing locus of the vertical map in the middle is $T(\Sigma_{i+1})$.
The indicated map vanishes by definition of $T(\Sigma_i)$, so
$\T_1(\I_{\Sigma_i}(\Theta))$ is isomorphic to the invertible sheaf
$\res{\widehat{\O(\Theta)}}{T(\Sigma_i)}$,
and the vanishing locus of
$\T_1(\I_{\Sigma_i}(\Theta)) \to \res{\P_x}{T(\Sigma_i)}$
is precisely $T(\Sigma_{i+1})$.
This proves Claim \ref{cl:T1}, which concludes the proof of the Lemma.
\end{proof}

\section{Schottky}\label{sec:schottky}

In this section we prove part (1) of Theorem \ref{thm} by constructing
many trisecants to the Kummer variety of $A$.

\begin{prop}\label{prop:trinotation}
Let 
$\Gamma_{g-1} \subset \Gamma_g \subset \Gamma_{g+1} \subset \Gamma_{g+2}$
be finite subschemes of $A$ of degrees indicated by the subscripts, and assume
$\Gamma_{g+2}$ is theta-general and superabundant.
Let $S$ be the residual
scheme of $\Gamma_{g-1}$ in $\Gamma_{g+1}$ and $a$, $b$ the residual points of $\Gamma_{g-1}$ in $\Gamma_{g}$,
$\Gamma_{g}$ in $\Gamma_{g+1}$. Then, for every pair of closed points $y,y'\in T(\Gamma_{g-1})\setminus T(\Gamma_g)$ we
have
\begin{equation}\label{eq:tri}
\Theta_{a-y} \cap \Theta_{b-x} \subseteq T(S)_{-y} \cup \Theta_{a-y'},
\end{equation}
where $x$ is the only closed point in $T(\Gamma_g,\Gamma_{g+1})$, and where
the union on the right is defined as a scheme by the product of the corresponding ideals.
\end{prop}

\begin{figure}
\input minors.pspdftex
\caption{Matrices in the proof of Proposition \ref{prop:trinotation}}
\label{fig:minors}
\end{figure}

\begin{proof}
We will first establish the schematic inclusion
\begin{equation}\label{eq:kincl}
\Theta_{a-y} \cap \Delta(\Gamma_{g+1},\Gamma_{g-1}) \subseteq
T({S})_{-y} \cup \Delta(\Gamma_g, \Gamma_{g-1})
\end{equation}
and afterwards deduce \eqref{eq:tri} from this.

Since
$\Gamma_{g-1}\subset\Theta_y$ we have, by Remark
\ref{rem:resunion}, inclusions
$\sheaf{I}_S(-\Theta_y) \subset
\sheaf{I}_{\Gamma_{g+1}}$
and $\sheaf{I}_a(-\Theta_y) \subset
\sheaf{I}_{\Gamma_g}$. These give rise to a
commutative diagram
\begin{small}
\begin{equation*}
\xymatrix@!C0@R=6ex@C=3.5em{
& 0 \ar[rr] && \I_{\Gamma_{g+1}} \ar[rr]\ar@{^{(}->}'[d][dd] && \I_{\Gamma_{g-1}} \ar[rr]\ar@{=}'[d][dd] && \O_S
\ar[rr]\ar@{->>}'[d][dd] && 0 \\
0 \ar[rr] && \I_S(-\Theta_y) \ar[rr]\ar@{^{(}->}[dd]\ar@{^{(}->}[ru] && \O(-\Theta_y) \ar[rr]\ar@{=}[dd]\ar@{^{(}->}[ru]
 && \O_S \ar[rr]\ar@{->>}[dd]\ar@{=}[ru]  && 0 \\
& 0 \ar'[r][rr] && \I_{\Gamma_g} \ar'[r][rr] && \I_{\Gamma_{g-1}} \ar'[r][rr] && k(a) \ar[rr] && 0 \\
0 \ar[rr] && \I_{\set{a}}(-\Theta_y) \ar[rr]\ar@{^{(}->}[ru]  && \O(-\Theta_y) \ar[rr]\ar@{^{(}->}[ru]  && k(a)
\ar[rr]\ar@{=}[ru]  && 0
}
\end{equation*}
\end{small}
with exact rows.
Twist the diagram with $2\Theta$, use
that $2\Theta-\Theta_y$ is linearly equivalent to
$\Theta_{-y}$, and apply the Fourier-Mukai transform. This
produces the diagram
\begin{small}
\begin{equation*}
\xymatrix@!C0@C=1.5em@R=6ex{
&&& \F \ar[rrrr]^{f_1} \ar@{=}'[d][dd] &&&& \widehat{\O_S} \ar[rrrrr]\ar@{->>}'[d][dd] &&&&&
R^1\mathcal{S}(\sheaf{I}_{\Gamma_{g+1}}(2\Theta)) \ar[rrrrrr]\ar@{->>}'[d][dd] &&&&&&
R^1\mathcal{S}(\I_{\Gamma_{g-1}}(2\Theta)) \ar[rrrr] \ar[dd] &&&& 0 \\
\O(-\Theta_{-y}) \ar[rrrr]^{f_3} \ar@{=}[dd]\ar@{^{(}->}[rrru] &&&& \widehat{\O_S} \ar[rrrrr]\ar@{->>}[dd]\ar@{=}[rrru]
&&&&& \widehat{\sheaf{I}_S(\Theta_{-y})} \ar[rrrr]\ar@{->>}[dd]\ar[rrru] &&&& 0  \\
&&& \F \ar'[r][rrrr]^(0.3){f_2} &&&& \P_a \ar'[rr][rrrrr] &&&&& R^1\mathcal{S}(\sheaf{I}_{\Gamma_g}(2\Theta))
\ar[rrrrrr] &&&&&& R^1\mathcal{S}(\I_{\Gamma_{g-1}}(2\Theta)) \ar[rrrr] &&&& 0  \\
\O(-\Theta_{-y}) \ar[rrrr]^{f_4} \ar@{^{(}->}[rrru] &&&& \sheaf{P}_a \ar[rrrrr]\ar@{=}[rrru] &&&&&
\widehat{\sheaf{I}_{\set{a}}(\Theta_{-y})} \ar[rrrr]\ar[rrru] &&&& 0
}
\end{equation*}
\end{small}
with exact rows, where $\F$ is $\mathcal{S}(\I_{\Gamma_{g-1}}(2\Theta))$.
Replace $\F$ with a locally free sheaf admitting a surjection to $\mathcal{S}(\I_{\Gamma_{g-1}}(2\Theta))$.
Assume that $\O(-\Theta_{-y})$ is a direct summand of $\F$:
replace $\F$ with $\F\oplus\O(-\Theta_{-y})$ if necessary.
Then we have the above diagram, with $\F$ locally free,
and where the inclusions $\O(-\Theta_{-y})\subset \F$ on the left are split.

Now $\Delta(\Gamma_{g-1},\Gamma_{g+1})$ is locally defined by the maximal minors of $f_1$,
and $\Delta(\Gamma_{g-1},\Gamma_g)$, $T(S)_{-y}$ and $\Theta_{a-y}$ are the vanishing loci of $f_2$, $f_3$ and $f_4$,
respectively.
The inclusion \eqref{eq:kincl} can now be verified explicitly by locally representing $f_1$ by a matrix as indicated in
Figure \ref{fig:minors}.

By the Key Lemma \ref{lem:key}, the divisor $\Theta_{b-x}$ is contained in $\Delta(\Gamma_{g+1})$.
As $\Delta(\Gamma_{g-1})$ has codimension at least $2$, by Lemma \ref{lem:codim2},
it follows that $\Theta_{b-x}$ is contained in the closure of $\Delta(\Gamma_{g+1})\setminus\Delta(\Gamma_{g-1})$,
which in turn is contained in $\Delta(\Gamma_{g-1},\Gamma_{g+1})$.
Moreover, Lemma \ref{lem:deltatheta} shows that $\Delta(\Gamma_{g-1},\Gamma_g)$ is contained in $\Theta_{a-y'}$.
Thus the inclusion \eqref{eq:tri} follows from \eqref{eq:kincl}.
\end{proof}

Now we are ready to obtain a unidimensional family of trisecants in our principally polarized abelian variety.
\begin{prop}\label{prop:tri}
With notation and assumptions as in Proposition \ref{prop:trinotation}, fix a point
$y'\in T(\Gamma_{g-1})\setminus T(\Gamma_{g})$ with $y'+a \ne x+b$, and let $Y=S\cup\set{a+(x-y')}$.
Then we have the following set-theoretical inclusion,
\begin{equation*}
(-T(\Gamma_{g-1})\setminus T(\Gamma_{g}))_{\gamma}\subset V=\set{2\xi\mid \xi+Y\subset\psi^{-1}(l)\text{ for some line
}l\subset \PP^N}
\end{equation*}
where $\psi:A\to \PP^N$, with $N=2^g-1$, is the Kummer map, corresponding to the $\abs{2\Theta}$, and $\gamma=a+b-y'$.
\end{prop}
\begin{proof}
Since $y'\neq x\in T(\Gamma_g)$ and $b\neq a+(x-y')$, we have that $S$ does not contain $a+(x-y')$,
so $Y$ is well defined as a finite subscheme of degree three. We deal separately with the two possible cases.

\emph{Case i)}. $Y\cong\sum_{i=1}^3 \Spec k$, that is $a\neq b$.
 In this case, $Y=\set{a,b,a+(x-y')}$ and $T(S) = \Theta_a\cap \Theta_b$. So the inclusion \eqref{eq:tri} gives
\begin{equation*}
 \Theta_{a-y} \cap \Theta_{b-x} \subseteq \Theta_{b-y} \cup \Theta_{a-y'},
\end{equation*}
which implies that the three points
\begin{align*}
&\psi\left(a+\tfrac{1}{2}(-y-\gamma)\right),\\
&\psi\left(b+\tfrac{1}{2}(-y-\gamma)\right),\\
&\psi\left(a+(x-y')+\tfrac{1}{2}(-y-\gamma)\right)
\end{align*}
are collinear (see \cite[Prop. 11.9.3]{BL}),
where the factor $\frac{1}{2}$ means any counter image under the multiplication-by-two endomorphism of $A$.

\emph{Case ii)}. $Y\cong\Spec(k[\varepsilon]/\varepsilon^2)+\Spec k$, that is $S$ is a nonreduced scheme supported in
$a=b$.
The inclusion \eqref{eq:tri} can be written
 \begin{equation*}
\Theta\cap  \Theta_{y-x} \subseteq T(S)_{-a}\cup \Theta_{y-y'}.
 \end{equation*}
Let $s\in H^0(\O_\Theta(\Theta))$ be the section corresponding to $T(S)_{-a}$ (see Example \ref{ex:T(S)}). We will also
choose a nonzero section $\theta_t$ of $H^0(\O(\Theta_t))$ for all points $t$. Then $s\cdot\theta_{y-y'}$ vanish on
$\Theta\cap \Theta_{y-x}$. By the exact sequence,
\begin{equation*}
 0\to H^0(\O_{\Theta}(\Theta_{x-y'}))\stackrel{\theta_{y-x}}\to H^0(\O_{\Theta}(\Theta+\Theta_{y-y' })) \to
H^0(\O_{\Theta_{y-x} \cap \Theta}(\Theta+\Theta_{y-y'})),
\end{equation*}
we get that in $H^0(\O_{\Theta}(\Theta+\Theta_{y-y'})),$
\begin{equation*}
 \theta_{x-y'}\cdot\theta_{y-x}=(\text{const.})\theta_{y-y'}\cdot s.
\end{equation*} 
Thus, using \cite[proof of Thm 0.5, case ii)]{Wcrit} we get that, with $\gamma=2a-y'$,
the line $l$ passing through the two points
\begin{align*}
&\psi\left(a+\tfrac{1}{2}(-y-\gamma)\right),
&\psi\left(a+(x-y')+\tfrac{1}{2}(-y-\gamma)\right)
\end{align*}
is tangent at the former. More precisely $S + \frac{1}{2}(-y-\gamma) \subset \psi^{-1}(l)$.
\end{proof}

\begin{rem}\label{rem:objection}
An inclusion of the type $\Theta\cap\Theta_a \subseteq \Theta_b\cup \Theta_c$
must hold scheme theoretically to give a trisecant to the Kummer variety.
The argument \cite[Theorem 5.2]{PPschottky} is written set theoretically,
so we hope that our schematic treatment is clarifying.
\end{rem}

\begin{proof}[Proof of part (1) of Theorem \ref{thm}]
The Gunning-Welters criterion
(\cite{gunning} or \cite[Thm 0.5]{Wcrit})
states that, if the algebraic subset $V\subset A$ in Proposition \ref{prop:tri}
is at least one dimensional, then it is a smooth irreducible curve and $A$ is its Jacobian.
(This holds, according to Welters \cite[Remark 0.7]{Wcrit}, in arbitrary characteristic different from $2$,
when $Y$ is supported in at least two points, as is the case here.)
Since $T(\Gamma_{g-1},\Gamma_g)$ is at least one dimensional, Proposition \ref{prop:tri}
thus implies that $A$ is a Jacobian, and $T(\Gamma_{g-1},\Gamma_g)$ is an Abel-Jacobi curve.
\end{proof}

\begin{rem}\label{rem:YinAJ}
As a first step in the direction of the Castelnuovo statement, Theorem \ref{thm} (2), we note: in Proposition
\ref{prop:tri}, the residual subscheme $S$ of $\Gamma_{g-1}$ in $\Gamma_{g+1}$ is contained in a translate of the
algebraic set $V\subset A$.
In fact, by \cite[Rem.~0.6]{Wcrit}, the degree $3$ subscheme $Y$ is contained in a translate of $V$, and by definition
$Y$ contains $S$.
\end{rem}

\section{Castelnuovo}\label{sec:castelnuovo}

\begin{proof}[Proof of part (2) of Theorem \ref{thm}]
By part (1) of Theorem \ref{thm},
we know that $A=J(C)$ is the Jacobian of some curve $C$.
Let $i$ be maximal such that there exists a degree $i$ subscheme $\Sigma_i\subset\Gamma_{g+2}$,
which is contained in a translate of $\pm C$.
Replace $C$ with this translate of $\pm C$, and fix such $\Sigma_i\subset C$.
Then $i\ge 2$ by Remark \ref{rem:YinAJ}, and the claim is that $i=g+2$.
For contradiction, assume $i\le g+1$.
Then $\Sigma_i$ is contained in a degree $g+1$ subscheme $\Gamma_{g+1}\subset\Gamma_{g+2}$,
which we fix.

With notation as in Lemma \ref{lem:Tsum}, there is a schematic inclusion
\begin{equation}\label{eq:Tsum-again}
T(\Sigma_i, \Sigma_{i+1}) + T(\Lambda_j,\Lambda_{j+1}) \subseteq \Delta(\Gamma_{g+1}).
\end{equation}
where $j=g+1-i \le g-1$.
Now we use that $C$ and $W_{g-2}$ are theta-duals of each other
(Example \ref{ex:T(W)} with $C=W_1$).
Both $C$ and $W_{g-2}$ are a priori defined up to translation;
as we have fixed $C$, we may just as well fix $W_{g-2}$ as $T(C)$.
We infer that the inclusion $\Sigma_i\subset C$ is equivalent to $W_{g-2} \subseteq T(\Sigma_i)$,
so $T(\Sigma_i, \Sigma_{i+1})$ contains $W_{g-2}$.

If $i>2$, or equivalently $j<g-1$, then $T(\Lambda_j,\Lambda_{j+1})$ has dimension at least $2$.
Then $W_{g-2} + T(\Lambda_j,\Lambda_{j+1})$ necessarily equals all of $A$: clearly
$T(\Lambda_j, \Lambda_{j+1})$ has dimension at least $g-j \ge 2$,
and it is well known that all the loci $W_{g-k}$ are geometrically nondegenerate,
which implies that $W_{g-k}+Y=A$ for any $Y\subset A$ of dimension $k$ (see Ran \cite[\S II]{ran} and 
Debarre \cite[Prop.~1.4]{debarre}). Thus the left hand side in
\eqref{eq:Tsum-again} is $A$, which is a contradiction, since
$\Delta(\Gamma_{g+1})$ has codimension one.

The case $i=2$, $j=g-1$ remains.
By Example \ref{ex:TS_Jac}, the theta-dual $T(\Sigma_2)$ is the union $W\cup W'$ of two
translates of $\pm W_{g-2}$ (the two copies may coincide, when $C$ is hyperelliptic,
in which case $W\cup W'$ is to be understood as a multiplicity two scheme structure on $W$).
By minimality of $i$, neither $W$ nor $W'$ are contained in $T(\Sigma_3)$,
so $T(\Sigma_2,\Sigma_3)$ also equals $W\cup W'$.

By Proposition \ref{prop:tri} (and the proof of part (1) of Theorem \ref{thm}),
the locus $T(\Lambda_{g-1}, \Lambda_g)$ is a translate of $\pm C$.
Thus, the left hand side of \eqref{eq:Tsum-again} contains a translate of the divisor $W\cup W'\pm C$.
By the Key Lemma \ref{lem:key},
the divisorial part of the right hand side of \eqref{eq:Tsum-again} is a theta-translate.
So we have an inclusion
\begin{equation*}
(W\cup W')_c \pm C \subseteq \Theta
\end{equation*}
for some point $c$,
which says that $(W\cup W')_c$ is contained in $T(\pm C)$.
But the latter is just $\pm W_{g-2}$, which is integral,
so it cannot contain a translate of $W\cup W'$,
and we have a contradiction.
\end{proof}


\bibliographystyle{plain}
\bibliography{castelnuovo-schottky}

\end{document}